\documentclass[a4paper]{amsart}
\usepackage{graphicx}
\usepackage{amssymb}
\usepackage{amsmath}
\usepackage{amsthm}
\usepackage{amscd}
\usepackage[all,2cell]{xy}

\UseAllTwocells \SilentMatrices
\newtheorem{thm}{Theorem}[section]

\newtheorem{cor}[thm]{Corollary}
\newtheorem{lem}[thm]{Lemma}
\newtheorem{exm}[thm]{Example}

\newtheorem{prop}[thm]{Proposition}
\theoremstyle{definition}

\theoremstyle{remark}
\newtheorem{rem}[thm]{\bf Remark}
\numberwithin{equation}{section}

\usepackage[usenames,dvipsnames,svgnames,table]{xcolor}

\begin{document}
\title[Derived equivalences via HRS-tilting]
{Derived equivalences via HRS-tilting}
\author[Xiao-Wu Chen, Zhe Han, Yu Zhou]{Xiao-Wu Chen, Zhe Han$^*$, Yu Zhou}

\date{\today}

\thanks{E-mail: xwchen@mail.ustc.edu.cn, hanzhe0302@163.com, yuzhoumath@gmail.com}
\thanks{$^*$ the correspondence author}

\dedicatory{}%
\commby{}%

\begin{abstract}
Let $\mathcal{A}$ be an abelian category and $\mathcal{B}$ be the Happel-Reiten-Smal{\o} tilt of $\mathcal{A}$ with respect to a torsion pair. We give necessary and sufficient conditions for the existence of a derived equivalence between $\mathcal{B}$ and $\mathcal{A}$,  which is compatible with the inclusion of $\mathcal{B}$ into the derived category of $\mathcal{A}$. In particular, any splitting torsion pair induces a derived equivalence. We prove that for the realization functor of any bounded $t$-structure, its denseness implies its fully-faithfulness.
\end{abstract}

\maketitle

\section{Introduction}

Let $(\mathcal{T}, \mathcal{F})$ be a torsion pair in an abelian category $\mathcal{A}$. Let $\mathcal{B}$ be the corresponding HRS-tilt \cite{HRS}, which is a certain full subcategory  of the bounded derived category $\mathbf{D}^b(\mathcal{A})$. Moreover, it is the heart of a certain bounded $t$-structure on $\mathbf{D}^b(\mathcal{A})$, in particular, the category $\mathcal{B}$ is abelian. We denote by $G\colon \mathbf{D}^b(\mathcal{B})\rightarrow \mathbf{D}^b(\mathcal{A})$  the corresponding realization functor \cite{BBD, Bei}, that is, a triangle functor whose restriction on $\mathcal{B}$ coincides with the inclusion.

In general, this realization functor $G$ is not an equivalence. However, it is proved in \cite{HRS} that if the torsion pair $(\mathcal{T}, \mathcal{F})$ is tilting or cotilting, then $G$ is an equivalence; also see \cite{BV, Noo, Chen}.  We mention that  even if $(\mathcal{T}, \mathcal{F})$ is neither tilting nor cotilting, the realization functor  $G$ might still be an equivalence. Indeed, given a two-term tilting complex of modules, the corresponding torsion pair in the module category is in general neither tilting nor cotilting; compare \cite{HKM, BZ1}. The realization functor in this case is an equivalence, which might be taken as the derived equivalence induced by the tilting complex.

We mention that the HRS-tilting is very important in the representation theory of quasi-tilted algebras \cite{HRS} and in the derived equivalences between K3 surfaces \cite{Hu}. Moreover, it plays a central role in the study of Bridgeland's stability conditions \cite{Bri, Woo, Qiu, QW}.  Therefore, it is of great interest to know when the realization functor $G$ in the HRS-tilting is an equivalence, and thus yields a derived equivalence. The main result of this paper answers this question in full generality.

\vskip 10pt

\noindent {\bf Theorem A.}\quad \emph{Let $\mathcal{A}$ be an abelian category with a torsion pair $(\mathcal{T}, \mathcal{F})$. Denote by $\mathcal{B}$ the corresponding  HRS-tilt and let $G\colon \mathbf{D}^b(\mathcal{B})\rightarrow \mathbf{D}^b(\mathcal{A})$ be a realization functor. Then the following statements are equivalent:
\begin{enumerate}
\item The functor $G \colon \mathbf{D}^b(\mathcal{B})\rightarrow \mathbf{D}^b(\mathcal{A})$ is an equivalence;
\item The category $\mathcal{A}$  lies in the essential image of $G$;
        \item Each object $A\in \mathcal{A}$ fits into an exact sequence
        $$0\longrightarrow F^0\longrightarrow F^1\longrightarrow A\longrightarrow T^0 \longrightarrow T^1\longrightarrow 0$$
    with $F^i\in \mathcal{F}$ and $T^i\in \mathcal{T}$,  such that the corresponding class in the third Yoneda extension group ${\rm Yext}_\mathcal{A}^3(T^1, F^0)$ vanishes.
\end{enumerate}}

  \vskip 5pt

We point out two features of Theorem A:  (i) We unify the tilting and cotilting cases due to \cite{HRS} in a symmetric manner. Indeed, in the tilting case where $\mathcal{T}$ cogenerates $\mathcal{A}$, we  take $F^0=0=F^1$ in the exact sequence in (3). In the cotilting case, we take $T^0=0=T^1$. (ii) The characterization in (3) is intrinsic, since the HRS-tilt $\mathcal{B}$ and the functor $G$ are not explicitly involved. We emphasize that the Yext-vanishing condition in (3) is necessary; see Example~\ref{exm:vanish}.

We observe that Theorem A applies to splitting torsion pairs. Recall that a torsion pair $(\mathcal{T}, \mathcal{F})$ is splitting if ${\rm Ext}_\mathcal{A}^1(F, T)=0$ for any $T\in \mathcal{T}$ and $F\in \mathcal{F}$. In this situation, any object $A$ is isomorphic to $T\oplus F$ for some objects $T\in \mathcal{T}$ and $F\in \mathcal{F}$. Then for the exact sequence in (3), we take $F^0=0=T^1$ such that the middle short exact sequence splits.

We mention the related work \cite{SR, PV}, which studies when the realization functor for a general bounded $t$-structure is an equivalence. These work is related to Serre duality and tilting complexes, respectively.

By Theorem A(2), the denseness of $G$ implies its fully-faithfulness. Indeed, there is a general result for the realization functor of any bounded $t$-structure.

\vskip 10pt

\noindent {\bf Theorem B.}\quad \emph{Let $\mathcal{D}$ be a triangulated category with a bounded $t$-structure and its heart $\mathcal{A}$, and $G \colon \mathbf{D}^b(\mathcal{A})\rightarrow \mathcal{D}$ be its realization functor. Assume that $G$ is dense. Then $G$ is an equivalence. }

\vskip 5pt
We mention that the proof of Theorem B is somewhat routine. However, in view of \cite{COS}, the  assertion seems to be quite surprising.

The paper is structured as follows. In Section 2, we recall basic facts on $t$-structures, realization functors and torsion pairs. We study the canonical maps from the Yoneda extension groups in the heart to the Hom groups in the triangulated category. We prove that these canonical maps are compatible with $t$-exact functors; see Proposition \ref{prop:t-ex}. Then we prove Theorem B (= Theorem \ref{thm:2}). In Section 3, we divide the proof of Theorem A into three propositions. The key observation is Proposition \ref{prop:ff1}, where we show that the restriction of the realization functor to the backward HRS-tilt is fully faithful. We give various examples to obtain new derived equivalences in Section 4, which are related to TTF-triples and two-term silting subcategories.

\section{Preliminaries}

In this section, we recall basic facts on $t$-structures, realization functors and torsion pairs. We make preparation for the next section. For the realization functor of a bounded $t$-structure, we prove that its denseness implies its full-faithfulness.

\subsection{Canonical maps}\label{sec:canonial}

Let $\mathcal{A}$ be an abelian category. For two objects $X, Y\in \mathcal{A}$ and $n\geq 1$, ${\rm Yext}_\mathcal{A}^n(X, Y)$ denotes the $n$-th Yoneda extension group of $X$ by $Y$, whose elements are equivalent classes $[\xi]$ of exact sequences
$$\xi\colon 0\rightarrow Y\rightarrow E^{-n+1}\rightarrow \cdots \rightarrow E^{-1} \rightarrow E^0\rightarrow X\rightarrow 0.$$ For $[\xi]\in {\rm Yext}_\mathcal{A}^n(X, Y)$ and $[\gamma]\in {\rm Yext}^m_\mathcal{A}(Y, Z)$, the Yoneda product $[\gamma\cup \xi]\in {\rm Yext}_\mathcal{A}^{n+m}(X, Z)$ is obtained by splicing $\xi$ and $\gamma$.

Let $\mathcal{D}$ be a triangulated category, whose translation functor is denoted by $\Sigma$. We denote by $\Sigma^{-1}$ a quasi-inverse of $\Sigma$. Then the powers $\Sigma^n$ are defined for all integers $n$. For two full subcategories $\mathcal{X}, \mathcal{Y}$ of  $\mathcal{D}$, we denote by
$$\mathcal{X}\ast \mathcal{Y}=\{Z\in \mathcal{D}\; |\;\exists\; \mbox{exact triangle } X\rightarrow Z \rightarrow Y\rightarrow \Sigma(X) \mbox{ with } X\in \mathcal{X}, Y\in \mathcal{Y}\}.$$
The operation $\ast$ is associative by \cite[Lemme 1.3.10]{BBD}.

 Recall that  a \emph{$t$-structure} $(\mathcal{D}^{\leq 0}, \mathcal{D}^{\geq 0})$ on $\mathcal{D}$ consists of two full subcategories $\mathcal{D}^{\leq 0}$ and  $\mathcal{D}^{\geq 0}$ subject to the following conditions:
\begin{enumerate}
	\item $\Sigma\mathcal{D}^{\leq 0}\subset \mathcal{D}^{\leq 0}$ and $\mathcal{D}^{\geq 0}\subset \Sigma\mathcal{D}^{\geq 0}$;
	\item ${\rm Hom}_\mathcal{D}(\mathcal{D}^{\leq 0}, \Sigma^{-1}\mathcal{D}^{\geq 0})=0$, that is, ${\rm Hom}_\mathcal{D}(X,Y)=0$ for any $X\in \mathcal{D}^{\leq 0}$ and $Y\in \Sigma^{-1}\mathcal{D}^{\geq 0}$;
	\item For any object $Z$ in $\mathcal{D}$, there exists an exact triangle
	\[X\longrightarrow  Z\longrightarrow  Y\longrightarrow  \Sigma(X)\]
	with $X\in \mathcal{D}^{\leq 0}$ and $Y\in \Sigma^{-1}\mathcal{D}^{\geq 0}$.
\end{enumerate}
The \emph{heart} of a $t$-structure $(\mathcal{D}^{\leq 0}, \mathcal{D}^{\geq 0})$ is the full subcategory $\mathcal{A}=\mathcal{D}^{\leq 0}\cap \mathcal{D}^{\geq 0}$, which is an abelian category. We shall assume that the $t$-structure $(\mathcal{D}^{\leq0},\mathcal{D}^{\geq 0})$ is \emph{bounded}, which means that
\[\mathcal{D}=\bigcup_{i,j\in\mathbb{Z}} (\Sigma^i \mathcal{D}^{\leq0}\cap\Sigma^j\mathcal{D}^{\geq 0}).\]
 We observe that a bounded $t$-structure is determined by its heart. Indeed, we have
  $$\mathcal{D}^{\leq 0}=\bigcup_{n\geq 0} \Sigma^{n}(\mathcal{A})\ast \cdots \ast \Sigma(A)\ast \mathcal{A}, \quad \mbox{and  } \mathcal{D}^{\geq 0}=\bigcup_{n\geq 0} \mathcal{A}\ast \Sigma^{-1}(\mathcal{A})\ast \cdots \ast \Sigma^{-n}(\mathcal{A}).$$
  Denote by $H_\mathcal{A}^0\colon \mathcal{D}\rightarrow \mathcal{A}$ the corresponding cohomological functor. Set $H_\mathcal{A}^n=H_\mathcal{A}^0\Sigma^n$ for $n\in \mathbb{Z}$. For details, we refer to \cite[1.3]{BBD}.

In what follows, we assume that $\mathcal{D}$ has a bounded $t$-structure $(\mathcal{D}^{\leq 0}, \mathcal{D}^{\geq 0})$ with its heart $\mathcal{A}$. For any objects $X, Y\in \mathcal{A}$ and $n\geq 1$, we shall recall the construction of  the \emph{canonical maps}
\begin{align}
\theta^n=\theta_{X, Y}^n\colon {\rm Yext}_\mathcal{A}^n(X, Y)\longrightarrow {\rm Hom}_\mathcal{D}(X, \Sigma^n(Y)), \quad [\xi]\mapsto \theta^n(\xi).
\end{align}
For the case $n=1$, we take an exact sequence $\xi \colon 0\rightarrow Y\stackrel{f}\rightarrow E\stackrel{g}\rightarrow X\rightarrow 0$ in $\mathcal{A}$. It fits uniquely into an exact triangle $$Y\stackrel{f}\longrightarrow E\stackrel{g}\longrightarrow X \xrightarrow{\theta^1(\xi)} \Sigma(Y).$$
Moreover, the morphism $\theta^1(\xi)$ depends on the equivalence class $[\xi]$.

  For the general case, we assume that $\xi \in {\rm Yext}^{n+1}_\mathcal{A}(X, Y)$. Write $[\xi]=[\xi_1\cup \xi_2]$ with $[\xi_1]\in {\rm Yext}_\mathcal{A}^{1}(Z, Y)$ and $[\xi_2]\in {\rm Yext}_\mathcal{A}^{n}(X, Z)$. We define $\theta^{n+1}(\xi)=\Sigma^n(\theta^1(\xi_1))\circ \theta^n(\xi_2)$. We observe that the morphism  $\theta^{n+1}(\xi)$ does not depend on the choice of $\xi_1$ and $\xi_2$.

The following results are well known; compare \cite[Remarque 3.1.17]{BBD}.

\begin{lem}\label{lem:can}
Keep the notation as above. Then the following statements hold.
\begin{enumerate}
\item The map $\theta^1$ is an isomorphism, and $\theta^2$ is injective.
\item Assume that $\theta^n_{A, B}$ are isomorphisms for all objects $A, B$ in $\mathcal{A}$. Then $\theta^{n+1}$ is injective.
    \item For $n\geq 2$, a morphism $f\colon X\rightarrow \Sigma^n(Y)$ lies in the image of $\theta_{X, Y}^n$ if and only if $f$ admits a factorization $X\rightarrow \Sigma(X_1) \rightarrow \cdots \rightarrow \Sigma^{n-1}(X_{n-1})\rightarrow \Sigma^n(Y)$ with each $X_i\in \mathcal{A}$.
\end{enumerate}
\end{lem}

\begin{proof}
In (1), the first statement is well known, and the second one is a special case of (2).  The statement in (3) follows from the surjectivity of $\theta^1$.

For (2), we assume that $\theta^{n+1}(\xi)=0$ for $[\xi]\in {\rm Yext}^{n+1}_\mathcal{A}(X, Y)$. Take an exact sequence $\xi_1\colon 0\rightarrow Y\rightarrow E\stackrel{g}\rightarrow Z \rightarrow 0$ such that $[\xi]=[\xi_1\cup \xi_2]$ for some element $[\xi_2]\in {\rm Yext}_\mathcal{A}^n(X, Z)$. Then we have the following commutative diagram with exact rows.
\[\xymatrix@C=.75cm{
{\rm Yext}_\mathcal{A}^n(X, E) \ar[rr]^-{{\rm Yext}_\mathcal{A}^n(X, g)} \ar[d]^-{\theta^n_{X, E}}  && {\rm Yext}_\mathcal{A}^n(X, Z) \ar[rr]^-{[\xi_1\cup-]} \ar[d]^-{\theta^n_{X, Z}} && {\rm Yext}_\mathcal{A}^{n+1}(X, Y) \ar[d]^-{\theta^{n+1}_{X, Y}}\\
{\rm Hom}_\mathcal{D}(X, \Sigma^n(E)) \ar[rr]^-{{\rm Hom}_\mathcal{D}(X, \Sigma^n(g))} && {\rm Hom}_\mathcal{D}(X, \Sigma^n(Z)) \ar[rr]^{{\rm Hom}_\mathcal{D}(X, \Sigma^n\theta^1(\xi_1))} && {\rm Hom}_\mathcal{D}(X, \Sigma^{n+1}(Y))
}\]
Then using the facts that $\theta^n_{X, E}$ and $\theta^n_{X, Z}$ are isomorphisms,  we infer $[\xi]=0$ by a diagram chasing.
\end{proof}

\begin{exm}\label{exm:1}
{\rm Let $\mathcal{A}$ be an abelian category. Denote by $\mathbf{D}^b(\mathcal{A})$ its bounded derived category. An object $X$ in $\mathcal{A}$ corresponds to a stalk complex concentrated on degree zero, which is still denoted by $X$. This allows us to view $\mathcal{A}$ as a full subcategory of $\mathbf{D}^b(\mathcal{A})$. For $X, Y\in \mathcal{A}$ and $n\in \mathbb{Z}$, we set ${\rm Ext}^n_\mathcal{A}(X, Y)={\rm Hom}_{\mathbf{D}^b(\mathcal{A})}(X, \Sigma^n(Y))$. We observe that ${\rm Ext}_\mathcal{A}^n(X, Y)=0$ for $n<0$.

For a complex $X$, $H^n(X)$ denotes the $n$-th cohomology of $X$. Then $\mathbf{D}^b(\mathcal{A})$ has a canonical $t$-structure given by $\mathbf{D}^b(\mathcal{A})^{\leq 0}=\{X\in \mathbf{D}^b(\mathcal{A})\; |\; H^n(X)=0 \mbox{  for all } n>0\}$ and $\mathbf{D}^b(\mathcal{A})^{\geq 0}=\{X\in \mathbf{D}^b(\mathcal{A})\; |\; H^n(X)=0 \mbox{  for all } n<0\}$. Then $\mathcal{A}$ is identified with the heart of this $t$-structure, and the corresponding cohomological functor $H_\mathcal{A}^0$ coincides with $H^0$. In this case, the canonical maps are denoted by \begin{align}\label{iso:chi}
\chi^n=\chi^n_{X, Y}\colon {\rm Yext}^n_\mathcal{A}(X, Y)\longrightarrow {\rm Ext}^n_\mathcal{A}(X, Y), \quad [\xi]\mapsto \chi^n(\xi).
\end{align}
They are isomorphisms for all $n\geq 1$; see \cite[Propositions XI.4.7 and 4.8]{Ive}. The morphism $\chi^n(\xi)\colon X\rightarrow \Sigma^n(Y)$ is known as the \emph{characteristic class} of $\xi$.}
\end{exm}

\subsection{$t$-exact functors}

Let $\mathcal{D}$ and $\mathcal{D}'$ be triangulated categories, whose translation functors are $\Sigma$ and $\Sigma'$, respectively. Recall that a triangle functor $(F, \omega)\colon \mathcal{D}\rightarrow \mathcal{D}'$ consists of an additive functor $F$ and a natural isomorphism $\omega\colon F\Sigma\rightarrow \Sigma' F$ such that any exact triangle $X\rightarrow Y\rightarrow Z\stackrel{h}\rightarrow \Sigma(X)$ in $\mathcal{D}$ is sent to an exact triangle $F(X)\rightarrow F(Y)\rightarrow F(Z) \xrightarrow{\omega_X\circ F(h)} \Sigma' F(X)$ in $\mathcal{D}'$. For $n\geq 1$, we define natural isomorphisms $\omega^n\colon F\Sigma^n\rightarrow \Sigma'^nF$ inductively: $\omega^1=\omega$ and $\omega^{n+1}=\Sigma' \omega^n \circ \omega\Sigma^n$. When the natural isomorphism $\omega$ is irrelevant, we will denote the triangle functor $(F, \omega)$ simply by $F$.

We assume that both $\mathcal{D}$ and $\mathcal{D}'$ have bounded $t$-structures, whose hearts are denoted by $\mathcal{A}$ and $\mathcal{A}'$, respectively. A triangle functor $F\colon \mathcal{D}\rightarrow \mathcal{D}'$ is \emph{$t$-exact} provided that $F(\mathcal{A})\subseteq \mathcal{A}'$.

\begin{lem}\label{lem:t-ex}
Let $F\colon \mathcal{D}\rightarrow \mathcal{D}'$ be a $t$-exact functor as above. Then the following statements hold.
\begin{enumerate}
\item The restriction $F|_\mathcal{A}\colon \mathcal{A}\rightarrow \mathcal{A}'$ is exact.
\item There are natural isomorphisms $H_{\mathcal{A}'}^n(F(X))\simeq F|_\mathcal{A}(H_\mathcal{A}^n(X))$ for $X\in \mathcal{D}$ and $n\in\mathbb{Z}$.
    \item If $F$ is an equivalence, then so is the restriction $F|_\mathcal{A}$.
\end{enumerate}
\end{lem}

\begin{proof}
Recall that a sequence  $0\rightarrow A\stackrel{u}\rightarrow B\stackrel{v}\rightarrow C\rightarrow 0$ in $\mathcal{A}$  is exact if and only if it fits into an exact triangle $A\stackrel{u}\rightarrow B\stackrel{v}\rightarrow C\rightarrow \Sigma(A)$ in $\mathcal{D}$; a similar remark holds for $\mathcal{A}'$. Then we infer (1).

Since $F$ is $t$-exact,  it commutes with the truncation functors in the sense \cite[1.3.3]{BBD}. Then (2) follows immediately, since $H^n_\mathcal{A}$ and $H^n_{\mathcal{A}'}$ are composition of these truncation functors. For (3), it suffices to show the denseness of $F|_\mathcal{A}$. Applying (2), we infer that if $F(X)$ lies in $\mathcal{A}'$, then $X$ necessarily lies in $\mathcal{A}$. This proves the required denseness.
\end{proof}

\begin{prop}\label{prop:t-ex}
Let $(F, \omega)\colon \mathcal{D}\rightarrow \mathcal{D}'$ be a $t$-exact functor as above. Then the following diagram commutes
\[\xymatrix{
{\rm Yext}_\mathcal{A}^n(X, Y) \ar[d]_-{\theta^n} \ar[rr]^-{F|_\mathcal{A}} && {\rm Yext}_{\mathcal{A}'}^n(F(X), F(Y))\ar[d]^-{\theta'^n}\\
{\rm Hom}_\mathcal{D}(X, \Sigma^n(Y))\ar[rr]^-{(F, \omega)} && {\rm Hom}_{\mathcal{D}'}(F(X), \Sigma'^nF(Y))
}\]
for any $X, Y\in \mathcal{A}$ and $n\geq 1$.
\end{prop}

In the diagram above, the columns are the canonical maps associated to the two $t$-structures, the upper row sends an exact sequence $\xi$  in $\mathcal{A}$ to the exact sequence $F|_\mathcal{A}(\xi)$ in $\mathcal{A}'$, and the lower row sends a morphism $f\colon X\rightarrow \Sigma^n(Y)$ to $\omega^n_Y\circ F(f)\colon F(X)\rightarrow \Sigma'^nF(Y)$.

\begin{proof}
Set $F'=F|_\mathcal{A}$. It suffices to prove the statement for the case $n=1$, since then the general case can be shown by induction.

Take an exact sequence $\xi \colon 0\rightarrow Y\stackrel{f}\rightarrow E\stackrel{g}\rightarrow X\rightarrow 0$ in $\mathcal{A}$. Then we have an exact triangle $Y\stackrel{f}\rightarrow E\stackrel{g}\rightarrow X\xrightarrow{\theta^1(\xi)} \Sigma(Y)$ in $\mathcal{D}$. Applying $(F, \omega)$ to it, we obtain the following exact triangle in $\mathcal{D}'$
$$F(Y) \stackrel{F(f)} \longrightarrow F(E) \stackrel{F(g)}\longrightarrow F(X) \xrightarrow{\omega_Y\circ F(\theta^1(\xi))} \Sigma'F(Y).$$
The morphisms $F(f)$ and $F(g)$ appear in the exact sequence $F'(\xi)$ in $\mathcal{A}'$. Hence, by the very definition of ${\theta'}^1$, we infer $\omega_Y\circ F(\theta^1(\xi))={\theta'}^1(F'(\xi))$, proving the statement in this case. Then we are done.
\end{proof}


We draw two consequences, Corollaries \ref{cor:1} and \ref{cor:2},  of Proposition \ref{prop:t-ex}.

\begin{cor}\label{cor:1}
Let $F\colon \mathbf{D}^b(\mathcal{A})\rightarrow \mathbf{D}^b(\mathcal{A}')$ be a triangle functor satisfying $F(\mathcal{A})\subseteq \mathcal{A}'$. Then $F$ is an equivalence if and only if so is the restriction $F|_\mathcal{A}\colon \mathcal{A}\rightarrow \mathcal{A}'$
\end{cor}

\begin{proof}
The ``only if" part follows from Lemma \ref{lem:t-ex}(3). For the ``if" part, we apply Proposition \ref{prop:t-ex}. In the commutative diagram, the upper row is an isomorphism induced by the equivalence $F|_\mathcal{A}$, and the columns are isomorphisms; see Example~\ref{exm:1}. It follows that $F$ induces isomorphisms
$${\rm Ext}^n_\mathcal{A}(X, Y)\longrightarrow {\rm Ext}^n_{\mathcal{A}'}(F(X), F(Y))$$
for all $X, Y\in \mathcal{A}$ and $n\geq 1$.  We observe that the isomorphisms hold also for $n\leq 0$. Recall that stalk complexes generate the bounded derived categories. Then we are done by \cite[Lemma II.3.4]{Hap}.
\end{proof}

In what follows, we assume that  $\mathcal{D}$ is  a triangulated category with a bounded $t$-structure $(\mathcal{D}^{\leq 0}, \mathcal{D}^{\geq 0})$  and its heart $\mathcal{A}$. By a \emph{realization functor}, we mean a triangle functor $(F, \omega)\colon \mathbf{D}^b(\mathcal{A})\rightarrow \mathcal{D}$ which is $t$-exact satisfying $F|_\mathcal{A}={\rm Id}_\mathcal{A}$; compare \cite[3.1]{BBD} and \cite[Appendix]{Bei}. Such a realization functor exists provided that $\mathcal{D}$ is algebraic, that is, triangle equivalent to the stable category of a Frobenius category; see \cite[3.2]{KV} and \cite[Section 3]{CR}.

\begin{rem}\label{rmk}
In \cite[Appendix]{Bei}, there is an explicit construction of a realization functor, which depends on the particular choice of  a filtered triangulated category over $\mathcal{D}$. The uniqueness of realization functors in general is not known. Indeed, as mentioned in \cite[Remark 4.9]{SR}, this uniqueness problem is intimately related to the open question whether any derived equivalences between algebras are standard; see the remarks after \cite[Definition 3.4]{Ric}. We mention that the realization functor becomes unique if  it is the base of a tower; compare \cite[Corollary 2.7]{Ke}.
\end{rem}

We have the following easy observation.

\begin{lem}\label{lem:real}
Let $F\colon \mathbf{D}^b(\mathcal{A})\rightarrow \mathcal{D}$  be a realization functor as above and $X\in \mathbf{D}^b(\mathcal{A})$. Then the following statements hold.
 \begin{enumerate}
 \item There are natural isomorphisms $H_\mathcal{A}^n(F(X))\simeq H^n(X)$ for $n\in \mathbb{Z}$.
 \item Assume that $F(X)\simeq A\in \mathcal{A}$. Then $X$ is isomorphic to $A$ in $\mathbf{D}^b(\mathcal{A})$.
 \end{enumerate}
\end{lem}

\begin{proof}
The first statement follows by Lemma \ref{lem:t-ex}(2). For the second one, we have $0=H_\mathcal{A}^n(F(X))\simeq H^n(X)$ for $n\neq 0$. Then $X$ lies in $\mathcal{A}$; moreover, it is isomorphic to $A$.
\end{proof}

The second part of the following result is standard; compare \cite[Proposition~3.1.16]{BBD}, \cite[Excercises IV.4.1 b)]{GM} and \cite[Theorem 4.7]{SR}.

\begin{cor}\label{cor:2}
Let $(F, \omega)\colon \mathbf{D}^b(\mathcal{A})\rightarrow \mathcal{D}$ be a realization functor as above.  Assume that $X, Y \in \mathcal{A}$ and $n\geq 1$. Then the following triangle
\begin{align}\label{equ:tri}
\xymatrix{
{\rm Ext}_\mathcal{A}^n(X, Y) \ar[rr]^-{(F, \omega)} && {\rm Hom}_\mathcal{D}(X, \Sigma^n(Y))\\
& {\rm Yext}^n_\mathcal{A}(X, Y) \ar[ur]_-{\theta^n} \ar[ul]^-{\chi^n}
}\end{align}
is commutative. Therefore, $F$ induces an isomorphism ${\rm Ext}_\mathcal{A}^1(X, Y)\rightarrow {\rm Hom}_\mathcal{D}(X, \Sigma(Y))$ and an injective map ${\rm Ext}_\mathcal{A}^2(X, Y)\rightarrow {\rm Hom}_\mathcal{D}(X, \Sigma^2(Y))$.

Consequently,  the following statements are equivalent:
\begin{enumerate}
\item The realization functor $F$ is full;
\item The realization functor $F$ is an equivalence;
\item The canonical maps $\theta^n$ are isomorphisms for all $n\geq 1$;
\item The canonical maps $\theta^n$ are surjective for all $n\geq 1$.
\end{enumerate}
\end{cor}

We observe that the condition (3) is independent of the choice of such a realization functor $F$. Hence, if one of the realization functors of the given $t$-structure is an equivalence, then all the realization functors are equivalences.

\begin{proof}
The triangle is commutative by Proposition \ref{prop:t-ex}, while the map $\chi^n$ is an isomorphism; see (\ref{iso:chi}). Then we apply Lemma \ref{lem:can}(1).

``(1) $\Rightarrow$ (2)"\quad By Lemma \ref{lem:real}(1), we infer that $F$ is faithful on objects. Then $F$ is fully faithful by \cite[p.446]{Ric89}. Since the heart $\mathcal{A}$ generates $\mathcal{D}$, it follows that $F$ is dense. The implication ``(4) $\Rightarrow$ (3)" follows by applying Lemma \ref{lem:can}(1)(2) repeatedly.

To complete the proof, it suffices to claim that (3) is equivalent to the fully-faithfulness of $F$. Indeed, by \cite[Lemma II.3.4]{Hap}, the latter is equivalent to the condition that the upper row of (\ref{equ:tri}) is an isomorphism. Since $\chi^n$ is an isomorphism,  the claim follows immediately.
\end{proof}

\begin{thm}\label{thm:2}
Let $\mathcal{D}$ be a triangulated category with a bounded $t$-structure and its heart $\mathcal{A}$, and $(F, \omega)\colon \mathbf{D}^b(\mathcal{A})\rightarrow \mathcal{D}$ be its realization functor. Assume that $F$ is dense. Then $F$ is an equivalence.
\end{thm}

\begin{proof}
By Corollary \ref{cor:2}(4) and Lemma~\ref{lem:can}(1), it suffices to prove that the canonical map $\theta_{X, Y}^n$ is surjective for each $X, Y\in \mathcal{A}$ and $n\geq 2$. Take a morphism $f\colon X\rightarrow \Sigma^n(Y)$ in $\mathcal{D}$. By Lemma \ref{lem:can}(3) and induction, it suffices to prove that $f$ admits a factorization $X\rightarrow \Sigma^{n-1}(C)\rightarrow \Sigma^n(Y)$ for some $C\in \mathcal{A}$.

By the denseness of $F$, we have an exact triangle in $\mathcal{D}$
$$X\stackrel{f} \longrightarrow \Sigma^n(Y) \stackrel{a}\longrightarrow F(Z) \longrightarrow \Sigma(X) $$
for some complex $Z\in \mathbf{D}^b(\mathcal{A})$. Applying the cohomological functor $H^0_\mathcal{A}$ to the triangle, we infer that $H_\mathcal{A}^i(F(Z))=0$ for $i\neq -1, -n$; moreover, $H^{-n}_\mathcal{A}(a)$ is an isomorhism. By Lemma \ref{lem:real}(1), we have $H^i(Z)=0$ for $i\neq -1, -n$. Hence, by truncation, we may assume that the complex $Z$ is of the following form
$$\rightarrow 0\rightarrow Z^{-n} \rightarrow \cdots \rightarrow Z^{-2} \rightarrow Z^{-1}\rightarrow 0\rightarrow \cdots$$
Denote by $p\colon Z\rightarrow \Sigma^n(Z^{-n})$ the canonical projection.

There is a unique morphism $y\colon Y\rightarrow Z^{-n}$ in $\mathcal{A}$ such that $\Sigma^n(y)$ equals the following composition
$$\Sigma^n(Y)\stackrel{a}\longrightarrow F(Z) \xrightarrow{F(p)} F(\Sigma^n(Z^{-n})) \xrightarrow{\omega^n_{Z^{-n}}} \Sigma^n(F(Z^{-n}))=\Sigma^n(Z^{-n}),$$
where $\omega^n\colon F\Sigma^n\rightarrow \Sigma^n F$ is the natural isomorphism induced by $\omega$.

We observe that  $H^{-n}(p)$ is mono. By the isomorphism in Lemma \ref{lem:real}(1),  the morphism $H^{-n}_\mathcal{A}(F(p))$ is mono. Since $H^{-n}_\mathcal{A}(a)$ is an isomorphism, it follows that $H^{-n}_\mathcal{A}(\Sigma^n(y))=y$ is mono. Then the monomorphism  $y$ fits into an exact triangle in $\mathcal{D}$
$$\Sigma^{-1}(C) \stackrel{b}\longrightarrow Y \stackrel{y} \longrightarrow Z^{-n} \longrightarrow C$$
for some $C\in \mathcal{A}$.  We observe that $\Sigma^n(y)\circ f=0$. It follows that $f$ factors through $\Sigma^n(b)$, as required.
\end{proof}

\subsection{Torsion pairs} Let $\mathcal{A}$ be an abelian category. Recall that a \emph{torsion pair} $(\mathcal{T}, \mathcal{F})$ consists of two full subcategories subject to the following conditions:
\begin{enumerate}
  \item ${\rm Hom}_\mathcal{A}(\mathcal{T}, \mathcal{F})=0$, that is, ${\rm Hom}_\mathcal{A}(T,F)=0$ for any $T\in \mathcal{T}$ and $F\in \mathcal{F}$;
  \item For any object $X$ in $\mathcal{A}$, there exists a short exact sequence
  \begin{align}\label{equ:tor}
   0\longrightarrow T\longrightarrow X\longrightarrow F\longrightarrow 0
  \end{align}
  with $T\in \mathcal{T}$ and $F\in \mathcal{F}$.
\end{enumerate}
We observe that the exact sequence in (\ref{equ:tor}) is unique up to isomorphism.

For an exact sequence $\xi\colon 0\rightarrow X\rightarrow E\rightarrow Y\rightarrow 0$ and a morphism $t\colon Y'\rightarrow Y$, we denote by $[\xi].t$ the equivalence class in ${\rm Yext}_\mathcal{A}^1(Y', X)$ obtained by the pullback of $\xi$ along $t$. Similarly, for a morphism $s\colon X\rightarrow X'$ we denote by $s.[\xi]$ the equivalence class in ${\rm Yext}_\mathcal{A}^1(Y, X')$ obtained by the pushout of $\xi$ along $s$.

The following two lemmas will be used in the next section.

\begin{lem}\label{lem:ff}
Let $\mathcal{A}$ and $\mathcal{A}'$ be abelian categories with torsion pairs $(\mathcal{T}, \mathcal{F})$ and $(\mathcal{T}', \mathcal{F}')$, respectively. Assume that $G\colon \mathcal{A}'\rightarrow \mathcal{A}$ is an additive functor satisfying the following conditions:
\begin{enumerate}
\item The functor $G$ is exact  satisfying $G(\mathcal{T}')\subseteq \mathcal{T}$ and $G(\mathcal{F}')\subseteq \mathcal{F}$;
    \item The restrictions $G|_{\mathcal{T}'}$ and $G|_{\mathcal{F}'}$ are fully faithful;
    \item For any objects $F'\in \mathcal{F}'$ and $T'\in \mathcal{T}'$, the functor $G$ induces a surjective map ${\rm Hom}_{\mathcal{A}'}(F', T')\rightarrow {\rm Hom}_{\mathcal{A}}(G(F'), G(T'))$ and an injective map ${\rm Yext}^1_{\mathcal{A}'}(F', T')\rightarrow {\rm Yext}^1_\mathcal{A}(G(F'), G(T'))$.
\end{enumerate}
Then $G$ is fully faithful.
\end{lem}

\begin{proof}
Let $X$ be an object in $\mathcal{A}$ such that $G(X)\simeq 0$. We apply $G$ to the exact sequence $0\rightarrow T'\rightarrow X\rightarrow F'\rightarrow 0$ with $T'\in \mathcal{T}'$ and $F'\in \mathcal{F}'$. By condition (1), $G(T')\in\mathcal{T}$ and $G(F')\in\mathcal{F}$, so we have  $G(T')\simeq 0$ and $G(F')\simeq 0$. By condition (2), we infer that $T'\simeq 0$ and $F'\simeq 0$. Hence, $X\simeq 0$. This proves that $G$ is faithful on objects. It suffices to prove that $G$ is full, since an exact functor which is full and  faithful on objects is necessarily faithful.

For the fullness of $G$, we take a morphism $g\colon G(X_1)\rightarrow G(X_2)$ in $\mathcal{A}$. Consider the exact sequence $\xi_i\colon 0\rightarrow T'_i\xrightarrow{a_i} X_i \xrightarrow{b_i} F'_i\rightarrow 0$ with $T'_i\in \mathcal{T}'$ and $F'_i\in \mathcal{F}'$ for $i=1, 2$. We have the following commutative diagram by ${\rm Hom}_\mathcal{A}(G(T'_1), G(F'_2))=0$.
\[\xymatrix{
0\ar[r] & G(T'_1) \ar@{.>}[d]_-{s'} \ar[r]^-{G(a_1)} & G(X_1) \ar[d]^-{g} \ar[r]^-{G(b_1)} & G(F'_1) \ar@{.>}[d]^-{t'} \ar[r] & 0\\
0\ar[r] & G(T'_2) \ar[r]^-{G(a_2)} & G(X_2) \ar[r]^-{G(b_2)} & G(F'_2) \ar[r] & 0
}\]
By condition (2), there exist $s\colon T'_1\rightarrow T'_2$ and $t\colon F'_1\rightarrow F'_2$ satisfying $G(s)=s'$ and $G(t)=t'$. The above commutative diagram implies that $G(s).[G(\xi_1)]=[G(\xi_2)].G(t)$, or equivalently, $G(s.[\xi_1])=G([\xi_2].t)$. By the injective map in condition (3), we infer that $s.[\xi_1]=[\xi_2].t$. This implies the existence of the following commutative diagram.
\[\xymatrix{
0\ar[r] & T'_1 \ar[d]_-{s} \ar[r]^-{a_1} & X_1 \ar@{.>}[d]^-{f} \ar[r]^-{b_1} & F'_1 \ar[d]^-{t} \ar[r] & 0\\
0\ar[r] & T'_2 \ar[r]^-{a_2} & X_2 \ar[r]^-{b_2} & F'_2 \ar[r] & 0
}\]
Comparing the two diagrams above, we infer that $g-G(f)=G(a_2)\circ h'\circ G(b_1)$ for some morphism $h'\colon G(F'_1)\rightarrow G(T'_2)$. By the surjective map in (3), we may assume that $h'=G(h)$ for some $h:F'_1\to T'_2$ in $\mathcal{A}'$. Then we have $g=G(f+a_2\circ h\circ b_1)$. This completes the proof.
\end{proof}

\begin{lem}\label{lem:equi}
Keep the same assumptions as above. Then $G\colon \mathcal{A}'\rightarrow \mathcal{A}$ is an equivalence if and only if both $G|_{\mathcal{T}'}\colon \mathcal{T}'\rightarrow \mathcal{T}$ and $G|_{\mathcal{F}'}\colon \mathcal{F}'\rightarrow \mathcal{F}$ are equivalences, and $G$ induces an isomorphism ${\rm Yext}^1_{\mathcal{A}'}(F', T')\rightarrow {\rm Yext}^1_\mathcal{A}(G(F'), G(T'))$ for any $T'\in \mathcal{T}'$ and $F'\in \mathcal{F}'$.
\end{lem}

\begin{proof}
For the ``only if" part, we only prove the denseness of $G|_{\mathcal{T}'}$. For any $T\in\mathcal{T}$, since $G$ is an equivalence, there exists $X\in\mathcal{A}'$ such that
$T\in \mathcal{T}$.
Consider the exact sequence $0\rightarrow T'\rightarrow X\rightarrow F'\rightarrow 0$ with $T'\in \mathcal{T}'$ and $F'\in \mathcal{F}'$. Applying $G$ to it, we obtain an epimorphism  $T\rightarrow G(F')$. But, $G(F')$ lies in $\mathcal{F}$ and thus ${\rm Hom}_\mathcal{A}(T, G(F'))=0$. This implies that $G(F')\simeq 0$ and $F'\simeq 0$. Hence, $X\simeq T'$, belonging to $\mathcal{T}'$.

For the ``if" part, it suffices to show that $G$ is dense. Take an object $X\in\mathcal{A}$ and consider the exact sequence (\ref{equ:tor}). We assume that $T=G(T')$ and $F=G(F')$ for $T'\in \mathcal{T}'$ and $F'\in \mathcal{F}'$. By the above isomorphism between the Yoneda extension groups, we obtain an extension of $F'$ by $T'$,  which is mapped by $G$ to (\ref{equ:tor}). In particular, the object $X$ lies in the essential image of $G$.
\end{proof}

\section{HRS-tilting}

We will divide the proof of Theorem A into three propositions.  Throughout this section, $\mathcal{A}$ is an abelian category with a torsion pair $(\mathcal{T}, \mathcal{F})$.  We denote by $\Sigma$ the translation functor on $\mathbf{D}^b(\mathcal{A})$.

By \cite[I.2]{HRS}, there is a unique bounded $t$-structure on $\mathbf{D}^b(\mathcal{A})$ with heart
$$\mathcal{B}=\{X\in \mathbf{D}^b(\mathcal{A})\; |\; H^{-1}(X)\in \mathcal{F}, H^0(X)\in \mathcal{T}, H^i(X)=0 \mbox{ for } i\neq -1, 0\}.$$
The abelian category $\mathcal{B}$ is called the (forward) \emph{HRS-tilt} of $\mathcal{A}$ with respect to the torsion pair $(\mathcal{T},\mathcal{F})$. By truncation, any object in $\mathcal{B}$ is isomorphic to a $2$-term complex $Y$ with $Y^i=0$ for $i\neq 0, -1$. Moreover, we have $\mathcal{B}=\Sigma(\mathcal{F})\ast \mathcal{T}$. It follows that  $(\Sigma(\mathcal{F}),\mathcal{T})$ is a torsion pair in $\mathcal{B}$. With respect to this torsion pair, we consider the backward HRS-tilt of $\mathcal{B}$:
$$\mathcal{A}'=\{Z\in \mathbf{D}^b(\mathcal{B})\; |\; H^0(Z)\in  \mathcal{T}, H^1(Z)\in \Sigma(\mathcal{F}), H^i(Z)=0 \mbox{ for } i\neq 1, 0\}.$$
We denote by $\Sigma_\mathcal{B}$ the translation functor on $\mathbf{D}^b(\mathcal{B})$. Hence, we have $\mathcal{A}'=\mathcal{T}\ast \Sigma_\mathcal{B}^{-1}(\Sigma(\mathcal{F}))$ in $\mathbf{D}^b(\mathcal{B})$. Set $\mathcal{T}'=\mathcal{T}$ and $\mathcal{F}'=\Sigma_\mathcal{B}^{-1}(\Sigma(\mathcal{F}))$. Then $(\mathcal{T}', \mathcal{F}')$ is a torsion pair in $\mathcal{A}'$.

We fix a realization functor $(G, \omega)\colon \mathbf{D}^b(\mathcal{B})\rightarrow \mathbf{D}^b(\mathcal{A})$ with respect to the heart $\mathcal{B}$. In particular, the restrictions $G|_\mathcal{T}$ and $G|_{\Sigma(\mathcal{F})}$ coincide with the inclusions of $\mathcal{T}$ and $\Sigma(\mathcal{F})$ in $\mathbf{D}^b(\mathcal{B})$, respectively. The natural isomorphism  $\omega\colon G\Sigma_\mathcal{B}\rightarrow \Sigma G$ induces the isomorphism
\begin{align}\label{iso:2}
t_F=(\Sigma^{-1}\omega_{\Sigma^{-1}_\mathcal{B}\Sigma(F)})^{-1}\colon G\Sigma_\mathcal{B}^{-1}\Sigma (F) \longrightarrow  \Sigma^{-1}G \Sigma(F)=F
\end{align}
for each $F\in \mathcal{F}$. Therefore, we have $G(\mathcal{F}')=\mathcal{F}$ and $G(\mathcal{T}')=\mathcal{T}$. By $\mathcal{A}'=\mathcal{T}'\ast \mathcal{F}'$, we infer that $G(\mathcal{A}')\subseteq \mathcal{A}$. In other words, the functor $G$ is $t$-exact, where $\mathbf{D}^b(\mathcal{B})$ is endowed with the $t$-structure given by the heart $\mathcal{A}'$ and $\mathbf{D}^b(\mathcal{A})$ has the canonical $t$-structure. Consequently, by Lemma \ref{lem:t-ex}(1), the restriction $G|_{\mathcal{A}'}\colon \mathcal{A}'\rightarrow \mathcal{A}$ is exact.

We have the main observation in this section, whose first assertion is inspired by \cite[Theorem 1.1(d)]{BZ1}. We refer to Subsection~\ref{sec:canonial} for the canonical map $\theta^2$ for the heart $\mathcal{B}$ in $\mathbf{D}^b(\mathcal{A})$.

\begin{prop}\label{prop:ff1}
Keep the notation as above. Then the exact functor $G|_{\mathcal{A}'}\colon \mathcal{A}'\rightarrow \mathcal{A}$ is fully faithful.

Moreover, the following statements are equivalent:   \begin{enumerate}
\item The functor $G|_{\mathcal{A}'}$ is an equivalence;
 \item The canonical map $\theta^2\colon {\rm Yext}^2_\mathcal{B}(\Sigma(F), T)\rightarrow {\rm Hom}_{\mathbf{D}^b(\mathcal{A})}(\Sigma(F), \Sigma^2(T))$ is an isomorphism for any $F\in \mathcal{F}$ and $T\in \mathcal{T}$;
     \item Any morphism $F\rightarrow \Sigma(T)$ in $\mathbf{D}^b(\mathcal{A})$ factors through some object in $\mathcal{B}$  for any $F\in \mathcal{F}$ and $T\in \mathcal{T}$.
\end{enumerate}
\end{prop}

\begin{proof}
We observe that $G|_{\mathcal{T}'}\colon \mathcal{T}'\rightarrow \mathcal{T}$ is the identity functor. By the isomorphism (\ref{iso:2}), the restriction $G|_{\mathcal{F}'}\colon \mathcal{F}'\rightarrow \mathcal{F}$ is an equivalence.

In what follows, we verify the condition (3) in Lemma \ref{lem:ff}.
For this, we take arbitrary objects $\Sigma_{\mathcal{B}}^{-1}\Sigma(F)\in \mathcal{F}'$ and $T\in \mathcal{T}'=\mathcal{T}$, where $F$ lies in $\mathcal{F}$. The following square is commutative.
\[\xymatrix{
{\rm Hom}_{\mathcal{A}'}(\Sigma_\mathcal{B}^{-1}\Sigma(F), T) \ar[d]_-{\Sigma_\mathcal{B}} \ar[rr]^-{G} && {\rm Hom}_\mathcal{A}(G\Sigma_\mathcal{B}^{-1}\Sigma(F), T)\\
{\rm Ext}_\mathcal{B}^1(\Sigma(F), T) \ar[r]^-{(G, \omega)} & {\rm Hom}_{\mathbf{D}^b(\mathcal{A})}(\Sigma(F), \Sigma(T)) \ar[r]^-{\Sigma^{-1}} &  {\rm Hom}_\mathcal{A}(F, T) \ar[u]^-{{\rm Hom}_\mathcal{A}(t_F, T)}
}\]
Applying Corollary \ref{cor:2}, the leftmost map in the lower row is an isomorphism. It follows that the upper row is an isomorphism.

The following diagram commutes, where $\theta_{\mathcal{A}'}^1$ is the canonical map associated to the heart $\mathcal{A}'$ and the map $\chi^1$ for $\mathcal{A}$ is an isomorphism; see Example \ref{exm:1}.
\[\xymatrix@C=0.2cm{
{\rm Yext}_{\mathcal{A}'}^1 (\Sigma_\mathcal{B}^{-1}\Sigma(F), T) \ar[d]_-{\theta^1_{\mathcal{A}'}} \ar[r]^-{G} & {\rm Yext}_\mathcal{A}^1(G\Sigma_\mathcal{B}^{-1}\Sigma(F), T) & \ar[l]_-{{\rm Yext}_\mathcal{A}^1(t_F, T)} {\rm Yext}_\mathcal{A}^1(F, T)\ar[d]^{\chi^1}\\
{\rm Hom}_{\mathbf{D}^b(\mathcal{B})}(\Sigma_\mathcal{B}^{-1}\Sigma(F), \Sigma_\mathcal{B}(T)) \ar[d]_-{\Sigma_\mathcal{B}}&& {\rm Ext}_\mathcal{A}^1(F, T)\ar[d]^-{\Sigma}\\
{\rm Ext}_\mathcal{B}^2(\Sigma(F), T) \ar[rr]^-{(G, \omega)} && {\rm Hom}_{\mathbf{D}^b(\mathcal{A})} (\Sigma(F), \Sigma^2(T))
}\]
By Corollary \ref{cor:2}, the lower row is injective. It follows that $G$ induces an injective map $${\rm Yext}_{\mathcal{A}'}^1 (\Sigma_\mathcal{B}^{-1}\Sigma(F), T)\longrightarrow {\rm Yext}_{\mathcal{A}}^1 (G\Sigma_\mathcal{B}^{-1}\Sigma(F), T).$$ Then Lemma \ref{lem:ff} applies to $G$.

For ``(1) $\Leftrightarrow$ (2)", we apply Lemma \ref{lem:equi}. We observe by the above diagram that the condition therein is equivalent to the one that the lower row is an isomorphism. By the commutative triangle in Corollary \ref{cor:2} applied to $G$, this is equivalent to the condition that $\theta^2$ is an isomorphism. Thus, we have the equivalence between (1) and (2). By Lemma \ref{lem:can}(1), $\theta^2$ is always injective. Then the equivalence between (2) and (3) follows from Lemma \ref{lem:can}(3).
\end{proof}

 In what follows, we characterize the essential image of the  functor $G|_{\mathcal{A}'}\colon \mathcal{A}'\rightarrow \mathcal{A}$. Here, we recall that  the \emph{essential image} ${\rm Im}\; F$ of a functor $F\colon \mathcal{C}\rightarrow \mathcal{D}$ means the full subcategory of $\mathcal{D}$  consisting of those objects $D$, which are isomorphic to $F(C)$ for some object $C$ in $\mathcal{C}$.

\begin{prop}\label{prop:ff2}
Keep the notation as above. Let $A$ be an object in $\mathcal{A}$. Then the following statements are equivalent:
\begin{enumerate}
\item The object $A$ belongs to ${\rm Im}\; G|_{\mathcal{A}'}$;
\item There is an exact triangle $A\rightarrow B^0\rightarrow B^1\rightarrow \Sigma(A)$ in $\mathbf{D}^b(\mathcal{A})$ with $B^i\in \mathcal{B}$;
    \item There is an exact sequence in $\mathcal{A}$
    $$0\longrightarrow F^0\longrightarrow F^1\longrightarrow A\longrightarrow T^0 \longrightarrow T^1\longrightarrow 0$$
    with $F^i\in \mathcal{F}$ and $T^i\in \mathcal{T}$ such that the corresponding class in ${\rm Yext}_\mathcal{A}^3(T^1, F^0)$ vanishes;
    \item The object $A$ belongs to ${\rm Im}\;  G$.
\end{enumerate}
\end{prop}

\begin{proof}

``(1) $\Rightarrow$ (2)"\quad Recall that any object $Z$ in $\mathcal{A}'$ is isomorphic to a $2$-term complex in $\mathbf{D}^b(\mathcal{B})$ supported in degrees zero and one. Hence, we have an exact triangle in $\mathbf{D}^b(\mathcal{B})$
\[Z\longrightarrow B^0\longrightarrow B^1\longrightarrow \Sigma_{\mathcal{B}}(Z)\]
with $B^i\in \mathcal{B}$. Assume that $A\simeq G(Z)$. Recall that $G|_\mathcal{B}$ coincides with the inclusion of $\mathcal{B}$ in $\mathbf{D}^b(\mathcal{A})$. Applying $G$ to the above triangle, we obtain the required one.

``(2) $\Rightarrow$ (3)" \quad Denote the morphism $B^0\rightarrow B^1$  in the given triangle by $f$.  Recall that $\mathcal{B}=\Sigma(\mathcal{F})\ast \mathcal{T}$. Then we have two exact triangles in the following diagram with $F^i\in \mathcal{F}$ and $T^i\in \mathcal{T}$.
\[\xymatrix{
\Sigma(F^0)\ar[r]\ar@{.>}[d]_-{\Sigma(a)} & B^0 \ar[d]^-{f} \ar[r] & T^0 \ar@{.>}[d]^-{b}\ar[r]^-{\partial^0} & \Sigma^2(F^0) \ar@{.>}[d]^-{\Sigma^2(a)}\\
\Sigma(F^1)\ar[r] & B^1 \ar[r] & T^1 \ar[r]^-{\partial^1} & \Sigma^2(F^1)
}\]
Since ${\rm Hom}_{\mathbf{D}^b(\mathcal{A})}(\Sigma(F^0), T^1)=0$, we have the morphisms $a:F^0\to F^1$ and $b:T^0\to T^1$ in $\mathcal{A}$ making the diagram commute.  We might identify $T^i$ with $H^0(B^i)$ and thus $b$ with $H^0(f)$. Similarly, we identify $a$ with $H^{-1}(f)$.

On the other hand, by applying the usual cohomological functor to the given triangle in (2), we infer that $H^{-1}(f)$ is mono and $H^0(f)$ is epic. Hence, we conclude that $a$ is mono and $b$ is epic. We observe that such morphisms $a$ and $b$ are uniquely determined by $f$, as $(\Sigma(F), \mathcal{T})$ is a torsion pair in $\mathcal{B}$.

We apply the $3\times 3$ Lemma to get the following diagram, where the square in the southeast corner is anti-commutative.
\begin{align}\label{equ:33}
\xymatrix{
X \ar[r]^-{c}\ar[d]_-{\chi^1(\rho_3)} & A\ar[d] \ar[r]^-{c'} & Y\ar[d]^-{b'} \ar[r]^-{\chi^1(\rho_2)}  & \Sigma(X) \ar[d] \\
\Sigma(F^0)\ar[r]\ar[d]_-{\Sigma(a)} & B^0 \ar[d]^-{f} \ar[r] & T^0 \ar[d]^-{b}\ar[r]^-{\partial^0} & \Sigma^2(F^0) \ar[d]\\
\Sigma(F^1)\ar[r] \ar[d]_-{\Sigma(a')} & B^1 \ar[r] \ar[d] & T^1 \ar[r]^-{\partial^1} \ar[d]^-{\chi^1(\rho_1)} & \Sigma^2(F^1)\ar[d]^-{\Sigma^2(a')}\\
\Sigma(X) \ar[r]  & \Sigma(A) \ar[r]  & \Sigma(Y) \ar[r]^-{\Sigma\chi^1(\rho_2)}  & \Sigma^2(X)
}
\end{align}
Since $a$ is mono, we infer that $X\in \mathcal{A}$ with an exact sequence
$$\rho_3\colon 0\longrightarrow F^0\stackrel{a} \longrightarrow F^1\stackrel{a'}\longrightarrow X\longrightarrow 0.$$
Similarly, we have that $Y\in \mathcal{A}$ with an exact sequence
$$\rho_1\colon 0\longrightarrow Y \stackrel{b'}\longrightarrow T^0 \stackrel{b} \longrightarrow T^1\longrightarrow 0.$$
Hence, the upper row is indeed given by an exact sequence
$$\rho_2\colon 0\longrightarrow X \stackrel{c}\longrightarrow A \stackrel{c'} \longrightarrow Y\longrightarrow 0.$$
We splice these three exact sequences to obtain the required one, whose class in ${\rm Yext}_\mathcal{A}^3(T^1, F^0)$ is $\chi^3[\rho_3\cup \rho_2 \cup\rho_1]$. Then we are done by the following identity
\begin{align*}
\chi^3(\rho_3\cup \rho_2 \cup\rho_1) &=\Sigma^2\chi^1(\rho_3)\circ \Sigma\chi^1(\rho_2)\circ \chi^1(\rho_1)\\
                                     &= -\Sigma^2\chi^1(\rho_3)\circ \Sigma^2(a')\circ \partial^1=0,
\end{align*}
where the second equality uses the anti-commutative square in the southeast corner of (\ref{equ:33}), and the last one uses the fact $\chi^1(\rho_3)\circ a'=0$ by  the exact triangle in the leftmost column.

``(3) $\Rightarrow$ (2)" \quad We break the long exact sequence in (3) into three short exact sequences $\rho_i$ as above. By the vanishing condition, we have $\Sigma^2\chi^1(\rho_3)\circ \Sigma\chi^1(\rho_2)\circ \chi^1(\rho_1)=0$. Hence, by the following exact triangle
$$\Sigma^2(F^0)\xrightarrow{\Sigma^2(a)}\Sigma^2(F^1) \xrightarrow{\Sigma^2(a')} \Sigma^2(X) \xrightarrow{\Sigma^2\chi^1(\rho_3)} \Sigma^3(F^0),$$
we have a morphism $\partial^1\colon T^1\rightarrow \Sigma^2(F^1)$ satisfying $\Sigma^2(a')\circ \partial^1=-\Sigma\chi^1(\rho_2)\circ \chi^1(\rho_1)$. Hence, we obtain the anti-commutative square in the southeast corner of (\ref{equ:33}). Now, by $3\times 3$ Lemma, we complete the square into (\ref{equ:33}). Then the middle vertical exact triangle is the required one.

``(2) $\Rightarrow$ (4)" \quad Denote the morphism $B^0\rightarrow B^1$ in (2) by $f$. The corresponding $2$-term complex $\cdots \rightarrow 0\rightarrow B^0\stackrel{f}\rightarrow B^1\rightarrow 0\rightarrow \cdots$ is denoted by $Z$. Then we have a canonical exact triangle in $\mathbf{D}^b(\mathcal{B})$
$$Z\longrightarrow B^0\stackrel{f}\longrightarrow B^1 \longrightarrow \Sigma(Z).$$
Applying $G$ to it, we infer that $G(Z)\simeq A$.

``(4) $\Rightarrow$ (1)" \quad Assume that $G(Z)\simeq A$ for $Z\in \mathbf{D}^b(\mathcal{B})$. Recall that $G$ is $t$-exact. It follows by Lemma \ref{lem:t-ex}(2) that $G|_{\mathcal{A}'}(H^n_{\mathcal{A}'}(Z))\simeq H^n(G(Z))=0$ for $n\neq 0$. Here, $H^n_{\mathcal{A}'}$ denotes the cohomological functor corresponding to the heart $\mathcal{A}'$ in $\mathbf{D}^b(\mathcal{B})$. By Proposition \ref{prop:ff1}, $G|_{\mathcal{A}'}$ is fully faithful. We infer that $H^n_{\mathcal{A}'}(Z)=0$ for $n\neq 0$, that is, $Z\in \mathcal{A}'$. Then we are done.
\end{proof}

\begin{prop}\label{prop:ff3}
Keep the notation as above. Then the following statements are equivalent:
\begin{enumerate}
\item The functor $G \colon \mathbf{D}^b(\mathcal{B})\rightarrow \mathbf{D}^b(\mathcal{A})$ is an equivalence;
\item The category $\mathcal{A}$ is contained in ${\rm Im}\; G$;
\item The restricted functor $G|_{\mathcal{A}'}\colon \mathcal{A}'\rightarrow \mathcal{A}$ is an equivalence.
\end{enumerate}
\end{prop}

\begin{proof}
The implication ``(1) $\Rightarrow$ (2)" is trivial. The equivalence between (2) and (3) follows from Proposition~\ref{prop:ff1} and Proposition \ref{prop:ff2}.

It remains to show ``(3) $\Rightarrow$ (1)". For this, we take a realization functor $$F\colon \mathbf{D}^b(\mathcal{A}')\longrightarrow \mathbf{D}^b(\mathcal{B})$$ of the heart $\mathcal{A}'$. Then the restriction of the composition $GF\colon \mathbf{D}^b(\mathcal{A}')\rightarrow \mathbf{D}^b(\mathcal{A})$ to $\mathcal{A}'\to\mathcal{A}$ coincides with $G|_{\mathcal{A}'}$, thus is an equivalence. By Corollary \ref{cor:1}, the composition $GF$ is an equivalence. We observe that $\mathcal{B}\subseteq {\rm Im}\; F$. Indeed, for any object $B\in \mathcal{B}$, there is an object $Z\in \mathbf{D}^b(\mathcal{A}')$ such that $GF(Z)\simeq B$. By Lemma~\ref{lem:real}(2), we infer that $F(Z)\simeq B$.

We consider the backward HRS-tilt of $\mathcal{A}'$ with respect to the torsion pair $(\mathcal{T}', \mathcal{F}')$
$$\mathcal{B}'=\{X\in \mathbf{D}^b(\mathcal{A}')\; |\; H^{0}(X)\in \mathcal{T}', H^1(X)\in \mathcal{F}', H^i(X)=0 \mbox{ for } i\neq 0, 1\},$$
and the corresponding realization functor $E\colon \mathbf{D}^b(\mathcal{B}') \rightarrow \mathbf{D}^b(\mathcal{A}')$. Then the restriction $F|_{\mathcal{B}'}\colon \mathcal{B}'\rightarrow \mathcal{B}$ is fully faithful by Proposition \ref{prop:ff1}. We have observed that $\mathcal{B}\subseteq {\rm Im}\; F$. By Proposition \ref{prop:ff2}, we infer that $F|_{\mathcal{B}'}$ is an equivalence. By the same reasoning as above, we infer that the composition $FE\colon \mathbf{D}^b(\mathcal{B}')\rightarrow \mathbf{D}^b(\mathcal{B})$ is also an equivalence. Since both $EF$ and $GF$ are equivalences, so is $G$.
\end{proof}

Combining these propositions, we obtain the main result.

\begin{thm}\label{thm:1}
Let $\mathcal{A}$ be an abelian category with a torsion pair $(\mathcal{T}, \mathcal{F})$. Denote by $\mathcal{B}$ the corresponding  HRS-tilt and let $G\colon \mathbf{D}^b(\mathcal{B})\rightarrow \mathbf{D}^b(\mathcal{A})$ be a realization functor. Then the following statements are equivalent:
\begin{enumerate}
\item The functor $G \colon \mathbf{D}^b(\mathcal{B})\rightarrow \mathbf{D}^b(\mathcal{A})$ is an equivalence;
\item The category $\mathcal{A}$  lies in ${\rm Im}\; G$;
    \item The canonical maps $\theta^2_{X, Y}\colon {\rm Yext}^2_\mathcal{B}(X, Y)\rightarrow {\rm Hom}_{\mathbf{D}^b(\mathcal{A})}(X, \Sigma^2(Y))$ are isomorphisms for any $X, Y\in \mathcal{B}$;
        \item Each object $A\in \mathcal{A}$ fits into an exact sequence
        $$0\longrightarrow F^0\longrightarrow F^1\longrightarrow A\longrightarrow T^0 \longrightarrow T^1\longrightarrow 0$$
    with $F^i\in \mathcal{F}$ and $T^i\in \mathcal{T}$ such that the corresponding class in ${\rm Yext}_\mathcal{A}^3(T^1, F^0)$ vanishes.
\end{enumerate}
\end{thm}

\begin{proof}
The equivalence between (1) and (2) is contained in Proposition \ref{prop:ff3}; moreover,  both statements are equivalent to the denseness of $G|_{\mathcal{A}'}\colon \mathcal{A}'\rightarrow \mathcal{A}$. We have ``(1) $\Rightarrow$ (3)" by Corollary \ref{cor:2}. The implication of ``(3) $\Rightarrow$ (2)" follows from Proposition \ref{prop:ff1}. For ``(4) $\Leftrightarrow$ (2)", we just apply Proposition \ref{prop:ff2}.
\end{proof}

\begin{rem}
(1) In view of Corollary \ref{cor:2}, the following fact seems to be somehow surprising: to verify the equivalence for the realization functor $G$, we only need to check the surjectivity of the second canonical map $\theta^2$.

(2) We observe that in the above proof , if $G$ is an equivalence, then the realization functor $F\colon \mathbf{D}^b(\mathcal{A}')\rightarrow \mathbf{D}^b(\mathcal{B})$ is also an equivalence. In other words, the torsion pair $(\Sigma(\mathcal{F}), \mathcal{T})$ in $\mathcal{B}$ satisfies the conditions in Theorem \ref{thm:1}.
\end{rem}

\section{Examples}

In this section, we will give some examples for Theorem \ref{thm:1}, which are related to TTF-triples and two-term silting subcategories.

The derived equivalences in Example \ref{exm:2} generalize the classical APR-reflection \cite{APR} and HW-reflection \cite{HW}.  In Example \ref{exm:3}, we construct a torsion pair in a module category, which is non-splitting, non-tilting and non-cotilting; moreover, it is not given by any two-term tilting complex. However, it does satisfy the conditions in Theorem \ref{thm:1}. In Proposition \ref{prop:2-term}, we apply Theorem \ref{thm:1} to two-term silting subcategories.

\subsection{TTF-triples and derived equivalences}

Let $\mathcal{A}$ be an abelian category. For a subcategory $\mathcal{U}$ of $\mathcal{A}$, denote by ${\rm Sub}\; \mathcal{U}$ (\emph{resp}. ${\rm Fac}\; \mathcal{U}$) the full subcategory consisting of subobjects (\emph{resp}. factor objects) of objects in $\mathcal{U}$. For two subcategories $\mathcal{U}, \mathcal{V}$ of $\mathcal{A}$, denote by $\mathcal{U}\ast \mathcal{V}$ the full subcategory consisting of those objects $Z$ such that there exists a short exact sequence $0\rightarrow  U\rightarrow Z\rightarrow V\rightarrow 0$ with $U\in\mathcal{U}$ and $V\in\mathcal{V}$.

\begin{cor}\label{cor:3}
Let $(\mathcal{T}, \mathcal{F})$ be a torsion pair in $\mathcal{A}$, and $\mathcal{B}$ be the corresponding HRS-tilt. Assume that either $\mathcal{A}=\mathcal{F}\ast ({\rm Sub}\; \mathcal{T})$ or $\mathcal{A}=({\rm Fac}\; \mathcal{F})\ast \mathcal{T}$ holds. Then any realization functor $G\colon \mathbf{D}^b(\mathcal{B})\rightarrow \mathbf{D}^b(\mathcal{A})$ is an equivalence.
\end{cor}

\begin{proof}
The exact sequence in Theorem \ref{thm:1}(4) exists, once we recall that $\mathcal{T}$ is closed under factor objects and that $\mathcal{F}$ is closed under subobjects.
\end{proof}

\begin{exm}
{\rm The above corollary includes the following cases.
\begin{enumerate}
\item[(1)] The torsion pair $(\mathcal{T}, \mathcal{F})$ is tilting (\emph{resp}. cotilting), which means $\mathcal{A}={\rm Sub}\; \mathcal{T}$ (\emph{resp}. $\mathcal{A}={\rm Fac}\; \mathcal{F}$). The existence of derived equivalences in these cases is due to \cite[Theorem I.3.3]{HRS}. For different approaches, we refer to \cite{BV, Noo, Chen}.
    \item[(2)] The torsion pair $(\mathcal{T}, \mathcal{F})$ is splitting, which means that ${\rm Ext}_\mathcal{A}^1(F, T)=0$ for any $F\in \mathcal{F}$ and $T\in \mathcal{T}$. In this case, any object $A$ in $\mathcal{A}$ is isomorphic to $F\oplus T$ for some $F\in \mathcal{F}$ and $T\in \mathcal{T}$. Then we have $\mathcal{A}=\mathcal{F}\ast \mathcal{T}$. The derived equivalence in this case seems to be new. We observe that it implies \cite[Proposition 5.7]{BZ1}.
        \end{enumerate}}
\end{exm}

Following \cite[Chapter IV.8]{S}, a triple $(\mathcal{X}, \mathcal{Y}, \mathcal{Z})$ of full subcategories  in $\mathcal{A}$ is called a \emph{TTF-triple},   provided that both $(\mathcal{X}, \mathcal{Y})$ and $(\mathcal{Y}, \mathcal{Z})$ are torsion pairs.

\begin{prop}\label{cor:4}
Let $(\mathcal{X}, \mathcal{Y}, \mathcal{Z})$ be a TTF-triple in $\mathcal{A}$. Then the following statements hold.
\begin{enumerate}
  \item The realization functor associated to the HRS-tilt of $\mathcal{A}$ with respect to $(\mathcal{X}, \mathcal{Y})$ is an equivalence if and only if $\mathcal{Z}\subseteq {\rm Sub}\; \mathcal{X}$.
  \item The realization functor associated to the HRS-tilt of $\mathcal{A}$ with respect to $(\mathcal{Y}, \mathcal{Z})$ is an equivalence if and only if $\mathcal{X} \subseteq{\rm Fac}\; \mathcal{Z}$.
\end{enumerate}
\end{prop}

\begin{proof}
We only prove (1), since the proof of (2) is similar. Assume that $\mathcal{Z}\subseteq {\rm Sub}\; \mathcal{X}$.
We have  $\mathcal{A}=\mathcal{Y}\ast\mathcal{Z}=\mathcal{Y}\ast({\rm Sub}\; \mathcal{X})$. Then the ``if" part follows from Corollary \ref{cor:3}.

Conversely, assume that $\mathcal{Z}\nsubseteq {\rm Sub}\; \mathcal{X}$. Take an object $A\in \mathcal{Z}$ such that it does not belong to ${\rm Sub}\; \mathcal{X}$. Then $A$ does not admit an exact sequence in Theorem \ref{thm:1}(4) for the torsion pair $(\mathcal{X}, \mathcal{Y})$, since ${\rm Hom}_\mathcal{A}(Y^1, A)=0$ for any $Y^1\in \mathcal{Y}$.
\end{proof}


In what follows, by an algebra $A$ we mean a finite dimensional algebra over a fixed field $k$. Denote by $A\mbox{-mod}$ the category of finite dimensional left $A$-modules.

\begin{exm}\label{exm:2}
{\rm Let $F\colon \mathcal{C}\rightarrow \mathcal{D}$ be a right exact functor between abelian categories. Denote by $\mathcal{A}$ the comma category of $F$. Recall that an object in $\mathcal{A}$ is  a triple $(C, D; \phi)$ with $C\in \mathcal{C}$, $D\in \mathcal{D}$ and $\phi\colon F(C)\rightarrow D$ a morphism in $\mathcal{D}$. The morphism $(f, g)\colon (C, D; \phi)\rightarrow (C', D'; \phi')$ consists of morphisms $f\colon C\rightarrow C'$ and $g\colon D\rightarrow D'$ subject to the condition $\phi'\circ F(f)=g\circ \phi$. We refer to \cite{FGR} for more details on comma categories.

Assume that $F$ is nonzero and  admits a right adjoint $G$. Then any morphism $\phi\colon F(C)\rightarrow D$ in $\mathcal{D}$ corresponds to a morphism $\phi^{\flat} \colon C\rightarrow G(D)$ in $\mathcal{C}$, called its \emph{adjoint}.

\begin{enumerate}
\item We view $\mathcal{C}$ and $\mathcal{D}$ as full subcategories of $\mathcal{A}$, by identifying $C\in\mathcal{C}$ with $(C,0; 0)$, and $D\in\mathcal{D}$ with $(0, D; 0)$, respectively. Denote by $\mathcal{E}$ the full subcategory of $\mathcal{A}$ consisting of those objects $(C, D; \phi)$ with $\phi$ an epimorphism. Denote by $\mathcal{M}$ the full subcategory consisting of those objects $(C, D; \phi)$ with its adjoint $\phi^\flat$ a monomorphism.   Then we have two TTF-triples  $(\mathcal{E}, \mathcal{D}, \mathcal{C})$ and $(\mathcal{D}, \mathcal{C}, \mathcal{M})$ in $\mathcal{A}$.  Denote by $\mathcal{B}_1$ (\emph{resp}. $\mathcal{B}_2$) the HRS-tilt of $\mathcal{A}$ with respect to $(\mathcal{E}, \mathcal{D})$ (\emph{resp}. $(\mathcal{C}, \mathcal{M})$).

We observe  $\mathcal{C}\subseteq \mathcal{E}$ and $\mathcal{D}\subseteq \mathcal{M}$. Then by Proposition \ref{cor:4},  we have the following  derived equivalences
$$\mathbf{D}^b(\mathcal{B}_1) \stackrel{\sim} \longrightarrow \mathbf{D}^b(\mathcal{A}) \stackrel{\sim}\longleftarrow \mathbf{D}^b(\mathcal{B}_2). $$

On the other hand, since $F$ is nonzero, we have $\mathcal{E}\nsubseteq \mathcal{C}={\rm Fac}\; \mathcal{C}$. Hence by Proposition \ref{cor:4}, any realization functor associated to the HRS-tilt of $\mathcal{A}$ with respect to $(\mathcal{D}, \mathcal{C})$ is not an equivalence. Indeed, the HRS-tilt is equivalent to $\mathcal{C}\times \mathcal{D}$, their direct product; compare \cite[Proposition I.2.3]{HRS}.

\item In the following concrete example, the above derived equivalences are well known and important in the representation theory of algebras.

Let $A$ be an algebra and $M_A$ be a nonzero right $A$-module. Denote by $\Gamma$ the one-point extension of $A$ by $M$, which by definition is the upper triangular matrix algebra
$\begin{pmatrix} k & M\\
0 & A\end{pmatrix}$. Denote by
$e=\begin{pmatrix} 1 & 0\\
                                                                                  0 & 0\end{pmatrix}$ the idempotent corresponding to $k$. Observe that the projective left $\Gamma$-module                                                                                  $\Gamma e$ is simple. We denote by $\Gamma_1$ the  corresponding APR-reflection \cite{APR} of $\Gamma$, and by $\Gamma_2$ the corresponding HW-reflection \cite{HW} of $\Gamma$.

Take $\mathcal{C}=A\mbox{-mod}$, $\mathcal{D}=k\mbox{-mod}$ and $F=M\otimes_A-$. Then the comma category $\mathcal{A}$ of $F$ is equivalent to $\Gamma\mbox{-mod}$. Applying \cite[Theorem 5.8]{HKM}, we observe that $\mathcal{B}_i$ is equivalent to $\Gamma_i\mbox{-mod}$ for $i=1,2$. In particular, we have the following derived equivalences
$$\mathbf{D}^b(\Gamma_1\mbox{-mod}) \stackrel{\sim} \longrightarrow \mathbf{D}^b(\Gamma\mbox{-mod}) \stackrel{\sim}\longleftarrow \mathbf{D}^b(\Gamma_2\mbox{-mod}). $$
We mention that the left equivalence is induced by the APR-tilting module, and the right one can be deduced from \cite[Theorem 10]{TW}; also see \cite{Lad}.

\end{enumerate}
}\end{exm}

In what follows, we construct an example for Theorem \ref{thm:1}, which seems to be not applied to any previously known results.  We say that a torsion class $\mathcal{U}$ in an abelian category $\mathcal{A}$ is \emph{finitely generated}, provided that there exists some object $Z\in \mathcal{U}$ which generates $\mathcal{U}$.

\begin{exm}\label{exm:3}
{\rm We keep the notation in Example \ref{exm:2}. In particular, the functor $F\colon \mathcal{C}\rightarrow \mathcal{D}$ is right exact and $\mathcal{A}$ denotes the comma category.

Assume that $(\mathcal{X}, \mathcal{Y})$ is a torsion pair in $\mathcal{C}$ and that $(\mathcal{U}, \mathcal{V})$ is a torsion pair in $\mathcal{D}$ such that $F(\mathcal{X})\subseteq \mathcal{U}$. We denote by $\mathcal{T}$ (\emph{resp}. $\mathcal{F}$) the full subcategory of $\mathcal{A}$ consisting of those objects $(C, D; \phi)$ with $C\in \mathcal{X}$ and $D\in \mathcal{U}$ (\emph{resp}. $C\in \mathcal{Y}$ and $D\in \mathcal{V}$). Then $(\mathcal{T}, \mathcal{F})$ is a torsion pair in $\mathcal{A}$.

    We assume that the following conditions are satisfied:
    \begin{enumerate}
    \item[(i)] The functor $F$ is exact with $F(\mathcal{C})\subseteq \mathcal{U}$;
    \item[(ii)] The torsion pair $(\mathcal{X}, \mathcal{Y})$ is tilting, non-splitting and non-cotilting;
    \item[(iii)] The torsion pair $(\mathcal{U}, \mathcal{V})$ is splitting and non-tilting such that $\mathcal{U}$ is not finitely generated.
    \end{enumerate}
    We claim that the resulted torsion pair $(\mathcal{T}, \mathcal{F})$ in $\mathcal{A}$ is non-splitting, non-tilting and non-cotilting such that $\mathcal{T}$ is not finitely generated;  moreover, it satisfies the condition in Theorem \ref{thm:1}(4).

    For the claim, it suffices to prove the last statement. We observe by (i) and (iii) that any object in $\mathcal{A}$ is isomorphic to $(C, U; \phi)\oplus(0, V; 0)$ with $C\in\mathcal{C}$, $U\in \mathcal{U}$ and $V\in \mathcal{V}$. It suffices to verify the condition for $(C, U; \phi)$. By (ii), we take an exact sequence $0\rightarrow C \rightarrow X^0\rightarrow X^1\rightarrow 0$ with $X^i\in \mathcal{X}$. By a pushout, we have the following commutative exact diagram.
    \[\xymatrix{
    0 \ar[r] & F(C)\ar[d]_-{\phi} \ar[r] & F(X^0) \ar@{.>}[d]\ar[r] & F(X^1) \ar@{=}[d] \ar[r] & 0\\
    0 \ar[r] & U \ar@{.>}[r] & U^0 \ar[r] & F(X^1) \ar[r] & 0
    } \]
    We observe that $U^0$ lies in $\mathcal{U}$. Then this yields the required exact sequence in $\mathcal{A}$.

    Using this claim, one can construct easily an indecomposable algebra $\Gamma$ such that there is a torsion pair $(\mathcal{T}, \mathcal{F})$ in $\Gamma\mbox{-mod}$ which is non-splitting, non-tilting and non-cotilting; moreover, it is not given by any $2$-term tilting complex; see \cite[Proposition 5.7(1)]{HKM}. However, it satisfies the condition in Theorem \ref{thm:1}(4). Consequently, the torsion pair induces a derived equivalence between $\Gamma\mbox{-mod}$ and its HRS-tilt.

    The construction of $\Gamma$ is similar to Example \ref{exm:2}(2). We take $A$ to be the path algebra given by a linear quiver, where such a torsion pair $(\mathcal{X}, \mathcal{Y})$ in $\mathcal{C}=A\mbox{-mod}$ is well known.  Let $B_1$ be a tame hereditary algebra and $B_2$ be an algebra with a simple injective module $S$. Set $B=B_1\times B_2$.  We take $\mathcal{D}=B\mbox{-mod}$, which is  identified  with $B_1\mbox{-mod}\times B_2\mbox{-mod}$. Set $\mathcal{U}=\mathcal{U}_1\times \mathcal{U}_2$,  where $\mathcal{U}_1$ is the additive subcategory  of $B_1\mbox{-mod}$ generated by preinjective $B_1$-modules and $\mathcal{U}_2$ the additive subcategory of $B_1\mbox{-mod}$ generated by $S$. The $B$-$A$-bimodule $M$ is taken such that $_BM$ lies in $\mathcal{U}$ and that $M_A$ is projective, and then $\Gamma$ is defined to be  the corresponding upper triangular matrix algebra. We omit the details.}
\end{exm}

\subsection{Two-term silting subcategories}

Torsion pairs arising from two-term silting complexes and subcategories were studied in different contexts such as abelian categories with arbitrary coproducts \cite{HKM, AMV} and Ext-finite abelian categories \cite{AIR, IJY, BZ2}. In what follows, we unify them into a general framework. This unification seems to be new, although it might be known to experts.

Let $\mathcal{A}$ be an abelian category. We say that a full additive subcategory $\mathcal{P}$ of $\mathbf{D}^b(\mathcal{A})$ is a \emph{two-term silting subcategory}, provided that the following conditions are satisfied:
\begin{enumerate}
\item The subcategory $\mathcal{P}$ is contravariantly finite in $\mathbf{D}^b(\mathcal{A})$;
\item (two-term) ${\rm Hom}_{\mathbf{D}^b(\mathcal{A})}(\mathcal{P}, \Sigma^i(\mathcal{A}))=0$ for $i\notin \{0, 1\}$;
\item (presilting) ${\rm Hom}_{\mathbf{D}^b(\mathcal{A})}(\mathcal{P},\Sigma^i(\mathcal{P}))=0$ for $i>0$;
\item (generating) If ${\rm Hom}_{\mathbf{D}^b(\mathcal{A})}(\mathcal{P}, \Sigma^i(M))=0$ for all $i$, then $M=0$.
\end{enumerate}
The first condition is necessary, because we do not assume that $\mathcal{A}$ has arbitrary coproducts or $\mathcal{A}$ is Ext-finite. In \cite{HKM, AIR, IJY}, the last condition is given in a slightly different manner, and one might consult \cite[Lemma~4.10 and Corollary~4.11]{BZ2} and \cite[Theorem~4.9]{AMV}.

The following results might be obtained in a very similar way as in \cite[Section~4]{BZ2}.

\begin{lem}\label{lem:two-term}
Let $\mathcal{P}\subseteq \mathbf{D}^b(\mathcal{A})$ be a two-term silting subcategory. Then the following statements hold.
\begin{enumerate}
\item We have a torsion pair $(\mathcal{T}(\mathcal{P}),\mathcal{F}(\mathcal{P}))$ in $\mathcal{A}$, where $\mathcal{T}(\mathcal{P})=\{X\in\mathcal{A}\mid {\rm Hom}_{\mathbf{D}^b(\mathcal{A})}(\mathcal{P},\Sigma(X))=0 \}$ and $\mathcal{F}(\mathcal{P})=\{X\in\mathcal{A}\mid {\rm Hom}_{\mathbf{D}^b(\mathcal{A})}(\mathcal{P},X)=0 \}$. The corresponding HRS-tilt $\mathcal{B}$ of $\mathcal{A}$ is given by
\begin{align}\label{eq:b}
\mathcal{B}=\{X\in {\mathbf{D}^b(\mathcal{A})}\mid {\rm Hom}_{\mathbf{D}^b(\mathcal{A})}(\mathcal{P},\Sigma^i(X))=0 \text{ for } i\neq 0\}.
\end{align}
\item For each $P\in \mathcal{P}$, there is an exact triangle
\begin{align*}
\Sigma(T) \stackrel{b}\longrightarrow P\longrightarrow \widetilde{P}\stackrel{a}\longrightarrow \Sigma^2(T)
\end{align*}
with $T\in\mathcal{T}(\mathcal{P})$ and $\widetilde{P}\in \mathcal{B}$; moreover, the objects $\{\widetilde{P}\mid P\in\mathcal{P}\}$ form a class of projective generators  in  $\mathcal{B}$. \hfill $\square$
\end{enumerate}
\end{lem}

We mention that one can show in a very similar way as in \cite[Section~4]{IY} that there is an equivalence $\mathcal{B}\stackrel{\sim}\longrightarrow{\rm mod}\mathcal{P}$, sending $B$ to ${\rm Hom}_{\mathbf{D}^b(\mathcal{A})}(-,B)|_{\mathcal{P}}$. Here,  ${\rm mod}\mathcal{P}$ denotes the category of contravariant finitely presented additive functors from $\mathcal{P}$ to the category of abelian groups.

The following result generalizes \cite[Theorem 1.1(e)]{BZ1} and the two-term version of \cite[Corollary 5.2]{PV}.

\begin{prop}\label{prop:2-term}
Let $\mathcal{P}$ be a two-term silting subcategory of $\mathbf{D}^b(\mathcal{A})$,  and $\mathcal{B}$ be the HRS-tilt of $\mathcal{A}$ with respect to  $(\mathcal{T}(\mathcal{P}),\mathcal{F}(\mathcal{P}))$. Then the realization functor $G\colon \mathbf{D}^b(\mathcal{B})\rightarrow \mathbf{D}^b(\mathcal{A})$ is an equivalence if and only if ${\rm Hom}_{\mathbf{D}^b(\mathcal{A})}(\mathcal{P},\Sigma^i(\mathcal{P}))=0$ for each $i<0$.
\end{prop}

\begin{proof}

By Theorem~\ref{thm:1}(3), it suffices to show that ${\rm Hom}_{\mathbf{D}^b(\mathcal{A})}(\mathcal{P},\Sigma^i(\mathcal{P}))=0$ for each $i<0$ if and only if the canonical maps $\theta^2_{X, Y}\colon {\rm Yext}^2_\mathcal{B}(X, Y)\rightarrow {\rm Hom}_{\mathbf{D}^b(\mathcal{A})}(X, \Sigma^2(Y))$ are surjective for any $X, Y\in \mathcal{B}$.

Suppose that ${\rm Hom}_{\mathbf{D}^b(\mathcal{A})}(\mathcal{P},\Sigma^i(\mathcal{P}))=0$ for each $i<0$. Then $\mathcal{P}$ is a subcategory of $\mathcal{B}$ by \eqref{eq:b}. Then $b=0$ in the triangle in Lemma~\ref{lem:two-term}(2), which implies $P\cong\widetilde{P}$. It follows that $\mathcal P$ consists of projective generators of $\mathcal B$. So for any object $X\in\mathcal{B}$, there is an exact sequence  $0\to Z\to P\to X\to 0$ in $\mathcal{B}$ with $P\in\mathcal{P}$, which induces an exact triangle $Z\to P\to X\to \Sigma(Z)$ in $\mathbf{D}^b(\mathcal{A})$. By \eqref{eq:b}, we have that any morphism from $X\to \Sigma^2(Y)$ with $Y\in\mathcal{B}$ factors through $\Sigma (Z)$. Hence by Lemma~\ref{lem:can}(3), the map $\theta^2_{X, Y}$ is surjective.

Conversely, we suppose that the canonical map $\theta^2_{X, Y}$ is surjective for any $X,Y\in\mathcal{B}$. It follows that for any $P\in\mathcal{P}$ and any $B\in\mathcal{B}$, we have
$${\rm Hom}_{\mathbf{D}^b(\mathcal{A})}(\widetilde{P},\Sigma^2(B))\cong{\rm Ext}^2_\mathcal{B}(\tilde{P},B)=0.$$
 Hence in the triangle in Lemma \ref{lem:two-term}(2), we have $a=0$. It follows that $\Sigma(T)$ is a direct summand of $P$. We have $\Sigma(T)\simeq 0$ since
 ${\rm Hom}_{\mathbf{D}^b(\mathcal{A})}(P, \Sigma(T))=0$. This proves that $\mathcal{P}\subseteq \mathcal{B}$. It implies by \eqref{eq:b} that ${\rm Hom}_{\mathbf{D}^b(\mathcal{A})}(\mathcal{P},\Sigma^i(\mathcal{P}))=0$ for each $i<0$.
\end{proof}

The following example shows the necessity of the Yext-vanishing condition in Theorem \ref{thm:1}(4).  For a set $\mathcal{S}$ of objects in an additive category, we denote by ${\rm add}\; \mathcal{S}$ the smallest additive subcategory which is closed under direct summands and  contains  $\mathcal{S}$.

\begin{exm}\label{exm:vanish}
{\rm Let $A$ be a Nakayama algebra given by the following  quiver
$$1\xrightarrow{a} 2\xrightarrow{b} 3\xrightarrow{c} 4\xrightarrow{d} 5\xrightarrow{e} 6$$
subject to the relations $cba=0=edc$. Denote by $P_i$ the indecomposable projective $A$-module corresponding to the vertex $i$. We have the following two-term silting complex $P$ supported on degrees $-1$ and $0$
$$P=(0\rightarrow P_1)\oplus (P_2\rightarrow P_1)\oplus (P_3\rightarrow P_1)\oplus (P_6\rightarrow P_4) \oplus (P_6\rightarrow P_5) \oplus (P_6\rightarrow 0),$$
 where the differentials of its indecomposable direct summands are the obvious morphisms. Then $\mathcal{P}={\rm add}\; P$ is a two-term silting subcategory in $\mathbf{D}^b(A\mbox{-mod})$. The corresponding torsion pair $(\mathcal{T}, \mathcal{F})$ in $A\mbox{-mod}$ is given by $\mathcal{T}={\rm add}\; \{1,\begin{smallmatrix}
1\\2
\end{smallmatrix},\begin{smallmatrix}
1\\2\\3
\end{smallmatrix},4, \begin{smallmatrix}
4\\5
\end{smallmatrix},5\}$ and $\mathcal{F}={\rm add}\; \{ 2, \begin{smallmatrix} 2\\ 3\end{smallmatrix}, 3, \begin{smallmatrix} 4\\ 5\\ 6\end{smallmatrix}, \begin{smallmatrix} 5\\ 6\end{smallmatrix}, 6\}$.
We observe ${\rm Hom}_{\mathbf{D}^b(A\mbox{-}{\rm mod})} (P, \Sigma^{-1}(P))\neq 0$.  By Proposition~\ref{prop:2-term}, the corresponding realization functor is not a derived equivalence. However, we have $A\mbox{-mod}=({\rm Fac}\; \mathcal{F})\ast ({\rm Sub}\; \mathcal{T})$. Consequently, for each $A$-module $X$, there is an exact sequence
$$0\longrightarrow F^0\longrightarrow F^1\longrightarrow X\longrightarrow T^0 \longrightarrow T^1\longrightarrow 0$$
    with $F^i\in \mathcal{F}$ and $T^i\in \mathcal{T}$, but the corresponding class in ${\rm Yext}^3_{A\mbox{-}{\rm mod}}(T^1, F^0)\simeq {\rm Ext}_A^3(T^1, F^0)$ might not vanish.}
\end{exm}

\vskip 5pt

\noindent {\bf Acknowledgements}\quad  The work is supported by National Natural Science Foundation of China (No.s 11522113, 11671245, 11671221 and 11626082). The third author was supported by FRINAT grant number 231000 from the Norwegian Research Council and is supported by YMSC. We thank Professor Bernhard Keller and Professor Bin Zhu for their encouragement and valuable comments.

\bibliography{}

\vskip 10pt

{\footnotesize \noindent Xiao-Wu Chen,\\
Key Laboratory of Wu Wen-Tsun Mathematics, Chinese Academy of Sciences,\\
School of Mathematical Sciences, University of Science and Technology of China, Hefei, 230026, P.R. China. \\
URL: http://home.ustc.edu.cn/$^\sim$xwchen\\

\noindent Zhe Han,\\
School of Mathematics and Statistics, Henan University,\\
475004 Kaifeng, P.R. China\\

\noindent Yu Zhou,\\
Yau Mathematical Sciences Center, Tsinghua University,\\
100084 Beijing, P.R. China}

\end{document}